\documentclass{article}
\usepackage{dsinclude}

\title{\textbf{Representation schemes and rigid maximal Cohen-Macaulay modules}}
\author{Hailong Dao and Ian Shipman}

\DeclareMathOperator{\MCM}{MCM}
\def\GV{G_{V_\bt}}
\def\RepRAV{\Rep_R(A,V_\bt)}

\begin{document}
\maketitle
\begin{abstract}
Let $\bk$ be an algebraically closed field and $A$ be a finitely generated, centrally finite, non-negatively graded (not necessarily commutative) $\bk$-algebra. In this note we construct a representation scheme for graded maximal Cohen-Macaulay $A$ modules. Our main application asserts that when $A$ is commutative with an isolated singularity, for a fixed multiplicity, there are only finitely many indecomposable rigid (i.e, with no nontrivial self-extensions) MCM modules up to shifting and isomorphism. We appeal to a result by  Keller, Murfet, and Van den Bergh to  prove a similar result for rings that are completion of graded rings. Finally, we discuss how finiteness results for rigid MCM modules are related to recent work by Iyama and Wemyss on maximal modifying modules over compound Du Val singularities.      
\end{abstract}

\section{Introduction}
Since the pioneering work of Kac \cite{Ka}, schemes parameterizing (framed) finite dimensional modules over algebras have been an important tool in representation theory \cite{Kra,Sc} and noncommutative geometry \cite{CEG,CQ,KR,L}.  Let $Q = (Q_0,Q_1)$ be a quiver with vertex set $Q_0$ and edge set $Q_1$.  Fix a field $\bk$.  A representation of $Q$ over $\bk$ is then a collection $\{V_i : i \in Q_0\}$ of $\bk$-vector spaces together with linear transformations $\phi_a:V_{s(a)} \to V_{t(a)}$ for each edge $a \in Q_1$ (where $s(a)$ and $t(a)$ are the source and target vertices of $a$, respectively).  The dimension vector of a representation $(\{V_i\},\{\phi_a\})$ is the vector $\bd = (\dim(V_i)) \in \N^{Q_0}$.  While the underlying vector spaces of a representation of $Q$ are not an isomorphism invariant, its dimension vector is one.  Let us fix a dimension vector $\bd$ and collection $\{V_i : i \in Q_0\}$ of finite-dimensional $\bk$-vector spaces where $\dim(V_\bt) = \bd$.  Then we define the representation space 
\[ \Rep(Q,V_\bt) = \prod_{a \in Q_1}{ \Hom_\bk(V_{s(a)},V_{t(a)}) }. \]
The group $G = \prod_{i \in Q_0} \GL(V_i)$ naturally acts on $\Rep(Q,V_\bt)$ by simultaneous change of basis.  The $G$-orbits are in bijective correspondence with the isomorphism classes of representations of $Q$ with dimension vector $\bd$.  An analogous construction is available to parameterize finite dimensional representations of any associative algebra, up to change of basis.  

Let $\bk$ be a field and suppose that $A$ is a $\bk$-algebra that is module-finite over its center $Z(A)$ and such that $Z(A)$ is a finitely generated $\bk$-algebra.  We adopt the following notion of maximal Cohen-Macaulay module in this setting, motivated by the notion of ``centrally Cohen-Macaulay'' in \cite{BHM,M}.

\begin{definition}  We say that a finitely generated, left $A$ module $M$ is \emphb{maximal Cohen-Macaulay (MCM)} if it is MCM over $Z(A)$.
\end{definition}

Suppose that $R \subset Z(A)$ is a polynomial subring over which $A$ is module-finite.  Let $M$ be an MCM $A$-module.  Then $M$ is free over $R$, so we can think about MCM $A$-modules as representations $A \to M_n(R)$ that are compatible with the action of $R$ on $A$.  This point of view goes back to the beginning of the study of MCM modules.  If MCM $A$-modules are analogous to finite dimensional representations of associative algebras, then what plays the role of the representation scheme? 

From now on we assume that $A$ is non-negatively graded and that $A_0$, the degree-zero component of $A$ is finite dimensional over $\bk$.  Once again, assume that $R \subset A$ is a graded, central, polynomial subring over which $A$ is module-finite.  (We do not assume that $R$ is standard-graded.)  Let $M$ be a graded MCM $A$-module.  Then $M$ is graded-free as an $R$-module.  We note that the isomorphism class over $R$ of a graded MCM $A$-module is constant in families.  So fix a graded $\bk$-vector space $V_\bt$.  We will construct a scheme $\RepRAV$ of finite type over $\bk$ parameterizing graded $A$-module structures on $V_\bt \tensor R$ which extend the free $R$-module structure.  Moreover, $\GV = \Aut_R(V_\bt \tensor R)_0$, the group of degree-preserving automorphisms of $V_\bt \tensor R$ over $R$ is an algebraic group and it acts on $\RepRAV$ by ``change of basis''.  The scheme $\RepRAV$ with its $\GV$ action shares many properties with its counterpart in the world of finite dimensional algebras.  As in the case of modules over a finite dimensional algebra, the $\GV$-orbits on $\RepRAV$ are in one-one correspondence with the isomorphism classes of MCM $A$-modules that are isomorphic to $V_\bt \tensor R$ as $R$-modules.  Furthermore for any point $M:A \to \End_R(V_\bt \tensor R)$ in $\RepRAV$, there is an exact sequence of $\bk$-vector spaces
\[ 0 \to \End_A(M)_0 \to \End_R(V_\bt \tensor R)_0 \to T_M \RepRAV \to \Ext^1_A(M,M)_0 \to 0, \]
where and $T_M \RepRAV$ is the Zariski tanget space to the scheme $\RepRAV$ at the point $M$ (see Theorem \ref{prop-tangent-sequence}).

One of the first applications of the representation scheme in the theory of finite dimensional algebras is to show that for each dimension vector there are only finitely many isomorphism classes of rigid modules, that is modules with no non-split self extensions (see \cite{Hap}). This has been applied in commutative algebra to show that for each multiplicity, there are only finitely many isomorphism classes of semi-dualizing modules (\cite{CS}). In this direction we prove:

\begin{mainthm}\label{mainThm}
Assume that $A$ is commutative, with an isolated singularity.  For each graded free $R$-module $V$ there are only finitely many isomorphism classes of rigid, graded MCM $A$-modules of type $V$.
\end{mainthm}

Together with a finiteness result about indecomposable modules, this implies
\begin{custcor}{A}\label{cor:finitely-many}
Up to shifting, there are only finitely many isomorphism classes of \emph{indecomposable}, rigid MCM $A$-modules of each rank.
\end{custcor}

Here, the rank is taken over some Noether normalization of $A$.  Obviously, one can replace that by the rank over $A$, if it is a domain, or more generally, by the Hilbert-Samuel multiplicity with respect to the maximal ideal.   

Next, we appeal to a result of Keller, Murfet, and Van den Bergh \cite{KMVDB} to obtain a finiteness theorem for gradable rings.
\begin{custcor}{B}\label{cor:complete}
Let $A$ be a complete, local $\bk$-algebra with an isolated singularity which is the completion of some non-negatively graded ring.  Then $A$ admits only finitely many indecomposable MCM $A$-modules of each rank.
\end{custcor}

In the final section, we consider the conjecture that a commutative local ring with an isolated singularity admits only finitely many isomorphism classes of rigid MCM modules of a given multiplicity.

\textbf{Acknowledgments}
We are delighted to thank Bhargav Bhatt and Igor Burban for interesting conversations and correspondence, and Srikanth Iyengar and Michael Wemyss for many helpful comments on an earlier version of this article.  The  authors are partially supported by NSF awards DMS-1104017 and  DMS-1204733.


\section{Main constructions}\label{sec:main-constructions}
Fix an algebraically closed field $\bk$.  Let $A$ be a non-negatively graded $\bk$-algebra whose center is finitely generated over $\bk$ and which is module-finite over its center.  Note that $Z(A)$ automatically inherits the grading from $A$.  Assume furthermore that $A_0$ is finite-dimensional over $\bk$ and that $Z(A)$ is a finitely generated $\bk$-algebra.  Throughout this section, unadorned tensor products are understood to be over $\bk$.

We recall from the introduction the class of modules that we will study.  A \emphb{graded maximal Cohen-Macaulay (MCM)} is a finitely generated, graded $A$ module whose restriction to $Z(A)$ is a graded MCM module.  There are several ways to characterize MCM modules over a commutative ring.  However in this paper, we will only need the following.  Suppose that $R \subset Z(A)$ is a graded polynomial subring of $Z(A)$ (not necessarily standard-graded) such that $Z(A)$ is a finitely generated module over $R$.  Then a $Z(A)$-module is MCM over if and only if it is free when viewed as an $R$-module.

Let $T$ be a commutative $\bk$-algebra.  Then we form the graded ring $A_T := T \tensor_\bk A$.  Observe that $Z(A_T) = T \tensor_\bk Z(A)$ and moreover, $A_T$ is a finitely generated module over $Z(A_T)$.  Indeed, if $a_1,\dotsc,a_r$ are generators for $A$ over $Z(A)$, then their images (also denoted) $a_1,\dotsc,a_r$ in $A_T$ generated $A_T$ over $Z(A_T)$.  We also note that we get a subring $R_T := T \tensor_\bk R \subset Z(A_T)$ over which $Z(A_T)$ (and $A_T$) are module-finite.  Since $R$ is a graded polynomial ring over $\bk$, $R_T$ is a graded polynomial ring over $T$.

Fix a finite-dimensional graded $\bk$-vector space $V_\bt$.  We consider $\End_\bk(V_\bt)$ as a graded ring where a linear transformation $\psi:V_\bt \to V_\bt$ is homogeneous of degree $m$ if and only if $\psi(V_r) \subset V_{r+m}$ for all $r$.

\begin{definition}
Let $T$ be a commutative $\bk$-algebra.  A \emphb{$T$-flat family of $V_\bt$-framed graded MCM $A$ modules} is a graded, $A_T$-module $M$ together with an isomorphism $V_\bt \tensor R_T \to M$ of graded $R_T$-modules.
\end{definition}

We now define a functor $\RepRAV$ from the category of commutative $\bk$-algebras to sets:
\[ \RepRAV(T) = \{\text{$T$-flat families of $V_\bt$-framed $A$-modules}\}. \]
Consider the algebraic group $\GV = \Aut_R(V_\bt \tensor R)_0$ of degree-preserving $R$-module automorphisms of the free, graded $R$-module $(V_\bt \tensor R)$.  This group fits into an exact sequence
\[ 0 \to \oplus_{i > j} \Hom_\bk(V_i, V_j \tensor R_{i-j}) \to \GV \to \prod_{i \in \Z} \GL(V_i) \to 1. \]

\begin{prop}\label{Rep-gr-modular}
The functor $\RepRAV$ is represented by an affine variety of finite type equipped with an action of the algebraic group $\GV$.  
\end{prop}
\begin{proof}
Fix a graded-free presentation of $A$ as an $R$ module 
\[ \xymatrix{ \bigoplus_{i=1}^s{ R(b_i) } \ar[r]^{\tau} & \bigoplus_{i=1}^r{ R(a_i) } \ar[r]^(.7)\sigma & A \ar[r] & 0.} \]
Consider the affine space
\[ X = \Hom_R( \oplus_{i=1}^r { R(a_i)}, \End_\bk(V_\bt) \tensor R )_0. \]
Note that $\GV$ naturally acts on $X$.  There is a morphism of functors
\[ \RepRAV \to X \]
where, by abuse of notation, we denote by $X$ both the functor represented by $X$ and $X$ itself.  Let $T$ be a test ring.  Then $X(T) = \Hom(\spec(T),X)$ is precisely the set of $R$-module maps 
\[ \oplus_{i=1}^r{R(a_i)} \to \End(V_\bt)\tensor_\bk R \tensor T. \] 
Suppose we are given an $R$-algebra morphism $\alpha:A \to \End_\bk(V_\bt)\tensor R \tensor T$.  By composing $\alpha$ with the presentation map we obtain an $R$-module morphism $\oplus_{i=1}^r{ R(a_i) } \to \End_\bt(V_\bt) \tensor R \tensor T$ and thus an element of $X(T)$.  It is clear that this construction is natural.  

Now, we claim that $\RepRAV \to X$ is a closed embedding.  Indeed, suppose given an $R$-module map
\[ \phi:\bigoplus_{i=1}^r{ R(a_i)} \to \End(V_\bt)\tensor R \tensor T \]
and a morphism $f:T \to T'$.   Notice that $\RepRAV(T') \to X(T')$ is injective.  So $f^*(\phi)$ is in its image if and only if the composite map
\[ f^*(\phi): \bigoplus_{i=1}^r{ R(a_i)} \to \End(V_\bt) \tensor R \tensor T' \]
factors through an $R$-algebra map
\[ A \to \End(V_\bt) \tensor R \tensor T'. \]
This amounts to two vanishing conditions.  First of all the composite map
\[ \bigoplus_{i=1}^s{ R(b_i)} \to \End(V_\bt) \tensor R \tensor T' \]
is zero.  Let $\beta_i$ be a generator of $R(b_i)$.  Then $f^*(\phi)$ factors through $A$ as an $R$-module map if and only if $\phi(\tau(\beta_i)) \in \End(V_\bt) \tensor R \tensor T$ maps to zero in $\End(V_\bt) \tensor R \tensor T'$.  Now, $\phi(\tau(\beta_i)) \in (\End(V_\bt) \tensor R)_{b_i} \tensor T$, which is a free summand of $\End(V_\bt)\tensor R \tensor T$.  Hence, there is a well defined subspace $W_i \subset T$ such that $f^*(\phi)(\tau(\beta_i)) = 0$ if and only if $f(W_i) = 0$. 

So assume that $f$ annihilates the ideal $(W_i : i = 1,\dotsc,s ) \subset T$.  Then $f^*(\phi)$ factors through an $R$-module map
\[ f^*(\phi): A \to \End(V_\bt)\tensor R \tensor T'. \]
For this to be an $R$-algebra map, it must satisfy two conditions.  First, $f^*(\phi)(1_A) = \id \tensor 1 \tensor 1$.  This means that $\phi(1_A) - \id \tensor 1 \tensor 1$ must map to zero under $f$.  Again there is a well defined subspace $U$ of $T$ such that $f^*(\phi)(1_A) = \id \tensor 1 \tensor 1$ if and only if $f(U) = 0$.  Second, we must have 
\begin{equation}\label{eq:multiplicative}
f^*(\phi)(a b) = f^*(\phi)(a)f^*(\phi)(b).
\end{equation}
Let $\alpha_i$ be a generator of $R(a_i)$ and suppose for each $i,j$ we have
\[ \sigma(\alpha_i) \sigma(\alpha_j) = \sigma( \sum{ c^l_{ij} \alpha_l } ) \]
Then \eqref{eq:multiplicative} holds if and only if 
\[ \phi(\sigma(\alpha_i))\phi(\sigma(\alpha_j)) - \phi( \sigma( \sum{ c^l_{ij} \alpha_l } ) ) \]
maps to zero under $f$.  Once again, each equation determines a subspace $U_{ij}$ such that the element above vanishes in $\End(V_\bt) \tensor R \tensor T'$ if and only if $f(U_{ij}) = 0$.

Putting all of these considerations together we find that $f^*(\phi) \in \RepRAV(T')$ if and only if $f$ annihilates the ideal $(W_i,U,U_{ij}) \subset T$.  Hence $\RepRAV$ is a closed subfunctor of $X$ and thus representable.  We notice that since $\GV$ acts on $X$ via algebra automorphisms of $\End(V_\bt) \tensor R$, $\RepRAV$ is preserved by the action of $\GV$.
\end{proof}

The following Proposition adapts \cite{V} (see \cite{G}) to our situation.

\begin{prop}\label{prop-tangent-sequence}
Let $p \in \RepRAV$ be a $\bk$-point.  Then there is an exact sequence
\[ 0 \to \End_A(M_p)_0 \to \End_R(V_\bt \tensor R)_0 \to T_p \RepRAV \to \Ext^1_A(M_p,M_p)_0 \to 0 \]
\end{prop}
\begin{proof}
We begin by observing that $\End_R(V_\bt \tensor R)_0 = \Lie( \GV )$ and take for the middle map the map associated to the action of $\GV$ on $\RepRAV$.  Explicitly, given an endomorphism $\phi \in \End_R(V_\bt \tensor R)_0$, we construct a framed $A[\epsilon]$-module as follows.  Let $\alpha:A \to \End(V_\bt)\tensor R$ be the action map.  Then the action of $A$ on $V_\bt \tensor R[\epsilon]$ corresponding to $\phi$ is
\begin{equation}\label{eq:action} a( v + \epsilon w ) = a v + \epsilon( \phi(av)-a\phi(v) + aw ). \end{equation}
Notice that if $\phi \in \End_A(M_p)$ then $\phi(av) - a \phi(v) = 0$.  Hence the action of $A$ on $V_\bt \tensor R[\epsilon]$ is split compatible with the framing.  Therefore the image of $\phi$ in $T_p \RepRAV$ is zero.

Turning to the second map, recall that the point $p$ represents an $A$-module structure on $V_\bt \tensor R$.  We denote this $A$-module by $M$.  By the modular description of $\RepRAV$ given in Proposition \ref{Rep-gr-modular}, we identify the space $T_p \RepRAV$ with the collection of $V_\bt$-framed MCM $A[\epsilon]$-modules (where $\epsilon^2 = 0$) $M_\epsilon$ such that $M_\epsilon/\epsilon M_\epsilon \cong M$ as framed $A$-modules.  Now, any such module fits into an exact sequence
\[ 0 \to \epsilon M_\epsilon \to M_\epsilon \to M_\epsilon / \epsilon M_\epsilon \to 0. \]
So we obtain a map $T_p \RepRAV \to \Ext^1_A(M_p,M_p)_0$.  

It remains to show that these maps give rise to an exact sequence.  First, suppose that 
\[ 0 \to M_p \to N \to M_p \to 0 \]
represents a given class $\eta \in \Ext^1_A(M_p,M_p)_0$.  Denote by $\epsilon$ the endomorphism obtained by composing
\[ N \to M_p \to N \]
the second, then first maps in the exact sequence representing $\eta$.  Since $\epsilon$ is an $A$-module map, this equips $N$ with the structure of an $A[\epsilon]$-module.  Since $M_p$ is free as a graded $R$-module, we can find an $R$-module splitting $s:M_p \to N$.  This induces a map $V_\bt \tensor R \to N$ and using the action of $\epsilon$, we obtain a framing $V_\bt \tensor R[\epsilon] \to N$.  By construction this reduces to the original framing modulo $\epsilon$.  Hence this framed $A[\epsilon]$-module represents an element of $T_p \RepRAV$ whose image in $\Ext^1_A(M_p,M_p)_0$ is $\eta$.  

Next, suppose that $A$ acts on $V_\bt \tensor R[\epsilon]$ in such a way that 
\[ 0 \to \epsilon V_\bt \tensor R \to V_\bt \tensor R[\epsilon] \to V_\bt \tensor R \to 0 \]
is split as an exact sequence of $A$-modules.  Let $s = \id + \epsilon \phi$ be a splitting, where we identify $V_\bt \tensor R[\epsilon] = V_\bt \tensor R \oplus \epsilon V_\bt \tensor R$.  Since $s$ is an $A$-module map we have
\[ av + \epsilon \phi(av) = s(av) = a \ast s(v) = a\ast v + \epsilon a \phi(v) \]
where we write $\ast$ for the action of $A$ on $V_\bt \tensor R[\epsilon]$.  Hence
\[ a \ast v = av + \epsilon( \phi(av) - a \phi(v) ). \]
Hence the element of $T_p \RepRAV$ corresponding to the action under consideration is in the image of $\End_R(V_\bt \tensor R)_0 = \Lie(\GV)$.

Finally, suppose that $\phi \in \End_R(V_\bt \tensor R)_0$ maps to zero in $T_p \RepRAV$.  Then the $A$-module structure on $V_\bt \tensor R[\epsilon]$ satisfies $a( v + \epsilon w) = av + \epsilon aw$.  Comparing this to \eqref{eq:action}, we find that $\phi(av) = a \phi(v)$ so that $\phi \in \End_A(V_\bt \tensor R)_0$.
\end{proof}

\begin{thm}
For each polynomial $H(t) \in \Q[t]$ there are finitely many isomorphism classes of rigid, graded, MCM $A$-modules with Hilbert polynomial $H(t)$.
\end{thm}
\begin{proof}
Two graded, MCM $A$-modules have the same Hilbert polynomial if and only if they are isomorphic over $R$.  Hence, the set of MCM $A$-modules with fixed Hilbert polynomial is parameterized by $\RepRAV$ for a particular $V_\bt$.  Since $\RepRAV$ has finite type over $\bk$, it has finitely many irreducible components.  Suppose that $p \in \RepRAV$ corresponds to a rigid $A$-module.  Then the infinitesimal action map $\Lie(\GV) \to T_p \RepRAV$ is surjective.  Hence, the orbit $\GV$ of $p$ is dense in its component.  We conclude that there can be only finitely many isomorphism classes of rigid MCM $A$-modules with a given Hilbert polynomial, at most the number of components of $\RepRAV$.
\end{proof}


\section{A finiteness theorem}
Let $A$ be a graded $\bk$-algebra as in Section \ref{sec:main-constructions}, with $R \subset Z(A)$ a graded subring over which $A$ is module-finite.  Furthermore, assume that $A$ is connected, so that the unit map $\bk \to A_0$ is an isomorphism.  Once again, unadorned tensor products are over the base field $\bk$.

Put $R_+ = \oplus_{m > 0} R_m$.  Given a finitely generated, graded $R$-module $M$ we define the quantities
\begin{align*}
g_{min}(M) & = \min\{m : (M/R_+M)_m \neq 0\}, \\
g_{max}(M) & = \max\{m : (M/R_+M)_m \neq 0\}, \\
w(M) & = g_{max} - g_{min}. 
\end{align*}
We extend these definitions to graded $A$-modules by viewing them as graded $R$-modules.  A graded MCM $A$-module is called \emphb{simple} if it does not admit any proper, nonzero MCM quotient modules.  Given an MCM $A$-module, graded or not we define $r(M) = \rank_R(M)$, the rank of $M$ as an $R$-module.  Since $A$ is module-finite over $R$ this quantity is finite.  Finally, for a graded $\bk$ vector space $V_\bt$ we say that a graded, MCM $A$-module $M$ has \emphb{type $V_\bt$} if $M/R_+ M \cong V_\bt$ as graded vector spaces.  The following finiteness result appears in Karroum's thesis.  
\begin{thm*}[\cite{K}]
Suppose that $A$ is commutative.  For each $r \geq 0$, there exists a natural number $\delta_r$ such that if $M$ is a simple, graded MCM $A$-module with $r(M) = r$ then $w(M) < \delta_r$.
\end{thm*}

While the result is stated for commutative rings, no substantial modification of the proof is needed to extend it to the non-commutative setting.  The main result of this section is to extend the previous theorem to indecomposable modules.  It does require that $A$ is commutative and furthermore that it has an isolated singularity.  We say that $A$ has an \emphb{isolated singularity} if for every prime ideal $\pfr \neq A_+$, the localization $A_\pfr$ is regular.  
\begin{thm}\label{thm:main.finiteness}
Assume that $A$ is commutative, with an isolated singularity.  For each $r > 0$ there exists $\alpha_r > 0$ such that if $M$ is an indecomposable, graded MCM $A$-module then $w(M) < \alpha_{r(M)}$.
\end{thm}

For the convenience of the reader and to introduce the ideas we will start by outlining a proof of Karroum's Theorem.  Suppose that $M$ is a simple graded MCM $A$-module.  Fix a set of homogeneous algebra generators $a_1,\dotsc,a_\gamma \in A$ for $A$ over $R$.  Define $\alpha = g_{max}(A)$ and note that $\alpha \geq \max\{ \deg(a_i) : i = 1,\dotsc,\gamma \}$.  Consider the finite dimensional, graded $\bk$-vector space $M/R_+ M$.  We note that $\dim_\bk(M/R_+ M) = r(M)$.  The following Lemma gives a way to produce MCM $A$-submodules of MCM $A$-modules.

\begin{lemma}\label{lem:submodule}
Suppose that there is some $i_0$ such that $(M/R_+M)_{i_0+j} = 0$ for $j = 1,\dotsc,\alpha$.  Let $M' = \sum_{i \leq i_0} R M_i$ be the $R$-submodule generated by the part of $M$ in degrees up to $i_0$.  Then $M'$ is preserved by the action of $A$ and both $M'$ and $M/M'$ are MCM.
\end{lemma}
\begin{proof}
We observe that the inclusion map $M' \to M$ induces an isomorphism $(M')_i \to M_i$ for $i \leq i_0 + \alpha$.  Let $m \in M'$ be a homogeneous.  Note that we we can express $m$ in terms of elements of bounded degree 
\[ m = \sum_{i\leq i_0}{ r_i m_i }, \quad r_i \in R, \quad m_i \in M'_i, \]
since $M'$ is generated in degress $i_0$ and less as an $R$-module.  Now to show that $M'$ is preserved by $A$ it suffices to show that $a_j M' \subset M'$ for $j=1,\dotsc,\gamma$.  Now we have
\[ a_j m = \sum_{i \leq i_0}{ r_i a_j m_i } \]
and $\deg(a_j m_i) \leq i_0 + \alpha$.  Hence $a_j m_i \in M'$.  The claims that $M'$ and $M/M'$ are MCM follow from the free-ness of $M'$ and $M/M'$.  Choosing an isomorphism $M \cong \oplus_i { (M/R_+)_i \tensor R }$ of $R$-modules, we find that $M' = \oplus_{i \leq i_0}{ (M/R_+ M)_i \tensor R}$ and $M/M' \cong \oplus_{i > i_0}{ (M/R_+ M)_i \tensor R}$. 
\end{proof}

To prove Karroum's theorem we may take $\delta_r = r \alpha + 1$.  Indeed, if $M$ is an MCM $A$-module with $w(M) > r(M) \alpha + 1$ then the pidgeonhole principle implies that there must exist $i_0$ such that $(M/R_+M)_{i_0+j} = 0$ for $j = 1,\dotsc, \alpha$.  Then Lemma \ref{lem:submodule} implies the  existence of an MCM submodule of $M$ with MCM quotient.  So $M$ is not simple.

Assume now that $A$ is commutative, with an isolated singularity.  Then for any two MCM $A$-modules $M,N$, the module $\Ext^1_{A}(M,N)$ is annihilated by a power of $A_+$.  In particular, $\Ext^1_A(M,N)$ is finite dimensional over $\bk$.  The main workhorse for this section is the following Lemma, which controls the largest nonzero graded component of $\Ext^1_A(M,N)$.

\begin{lemma}\label{lem:ext-degrees}
Given a natural numbers $r,s > 0$ there exists an integer $\beta_{r,s} \geq 0$ such that for all graded, MCM $A$-modules $M$ and $N$ with $r(M) \leq r, r(N) \leq s,$ and $g_{min}(M) > g_{max}(N) + \beta_{r,s}$
\[ \Ext^1_A(M,N)_0 = 0. \]
\end{lemma}
\begin{proof}
We proceed by induction on the pair $(r,s)$ (with the component-wise partial order).  During the induction, we will construct the $\beta$'s so that $\beta_{r',s'} \leq \beta_{r,s}$ if $(r',s') \leq (r,s)$.

Let $M,N$ be graded MCM $A$-modules with $r = \rk(M)$ and $s = \rk(N)$.  Consider an extension 
\[ 0 \to N \to E \to M \to 0. \]
Suppose that $M$ admits a proper, simple, graded MCM quotient $M \onto S$ with kernel $M' \subset M$.  Then $g_{min}(M'),g_{min}(S) \geq g_{min}(M)$.  So if $g_{min}(M) - g_{max}(N) \geq \beta_{r(M'),s},\beta_{r(S),s}$ then $\Ext^1_A(M',N) = \Ext^1_A(S,N)= 0$. From the exact sequence
\[ \Ext^1_A(S,N) \to \Ext^1_A(M,N) \to \Ext^1_A(M',N) \]
we see that $\Ext^1_A(M,N)=0$ as well.

Note that if $M$ is a graded, MCM $A$-module then so is $M^\vee := \Hom_R(M,R)$.  Moreover, there is a canonical isomorphism $\Ext^1(M,N) \cong \Ext^1(N^\vee,M^\vee)$.  Moreover, $g_{max}(M^\vee) = -g_{min}(M)$ and $g_{min}(N^\vee) = - g_{max}(N)$ so 
\[ g_{min}(M) - g_{max}(N) = g_{min}(N^\vee) - g_{max}(M^\vee). \]
So if $N$ is not simple, we can show that $\Ext^1_A(N^\vee,M^\vee) = 0$ by the argument in the previous paragraph and thus $\Ext^1_A(M,N) = 0$, provided $g_{min}(M) - g_{max}(N) > \max\{\beta_{r,s'} : s' \leq s\}$.

It remains to consider the case where $M$ and $N$ are simple.  Let $V_\bt = (M/R_+ M)(-g_{min}(M))$ and $W_\bt = (N/R_+ N)(-g_{min}(N))$.  We shift these graded vector spaces to normalize them so that the lowest nonzero component is in degree zero.  By Karroum's Theorem, $w(M),w(N) \leq w : =\max\{\delta_r,\delta_s\}$.  This means that there are only finitely many possibilities for the pair $(V_\bt,W_\bt)$.  In light of this we claim that it suffices to show that for any pair $V_\bt,W_\bt$ of graded vector spaces, there is a bound $m_{V_\bt,W_\bt}$ only depending on $V_\bt$ and $W_\bt$ such that $\Ext^1_A(M,N)_j = 0$ for all $M$ of type $V_\bt$, $N$ of type $W_\bt$, and $j$ with $|j| > m_{V_\bt,W_\bt}$.   Indeed, if $|j| > m_{V_\bt,W_\bt}$ then 
\[ \Ext^1_A(M(-g_{min}(M)),N(-g_{min}(N)))_j = \Ext^1_A(M,N)_{j + g_{min}(M)-g_{min}(N) } = 0. \]
Since $g_{min}(M) - g_{min}(N) \geq g_{min}(M) - g_{max}{N}$, it follows that if $g_{min}(M) > g_{max}(N) + m_{V_\bt,W_\bt}$ then $\Ext^1_A(M,N)_0 = 0$.   After making these reductions we see that 
\[ \beta_{r,s} = \max\{ \beta_{r',s'}, m_{V_\bt,W_\bt} : (r',s') < (r,s), \dim(V_\bt) = r, \dim(W_\bt) = s, w(V_\bt), w(W_\bt) \leq w \} \]
has the desired property.

Form $X = \Rep_A(V_\bt) \times \Rep_A(W_\bt),$ which is an affine scheme of finite type.  We have two tautological flat families of graded, MCM $A$-modules $\cM$ and $\cN$ via the projections.  We view these as graded sheaves on $X$, equipped with a homogenous action of $A \tensor \cO_X$.  

Let $\cF_\bt \to \cM$ be a resolution of $\cM$ where $\cF_i = \oplus_{j=1}^{b_i}{\cO_X \tensor A(a_j) }$.  Then for each $p \in X$, we obtain a graded-free resolution
\[ \cF_\bt|_p \to \cM|_p \]
of the graded $A$-module $\cM|_p$.  

By \cite{Y}, Proposition 6.17 (adapted to the graded setting), there exists a graded quotient ring $R \onto \Rbr$ with the following properties.  Put $\Abr = A \tensor_R \Rbr$.  Then $\dim_\bk(\Rbr) < \infty$ and for any graded, MCM $A$-modules $M,N$ the map
\[ \Ext^1_A(M,N) \to \Ext^1_{\Abr}(M \tensor_R \Rbr, N \tensor_R \Rbr) \]
is injective.  

Now for each $p \in X$, we note that $\cF_\bt|_p \tensor_R \Rbr \to \cM|_p \tensor_R \Rbr$ is a graded-free resolution of $\cM|_p \tensor_R \Rbr$.  Putting it all together, $\Ext^1_A(\cM|_p,\cN|_p)$ embeds in $\Ext^1_{\Abr}(\cM|_p \tensor_R \Rbr,\cN|_p \tensor_R \Rbr)$, which is a subquotient of
\begin{multline*} \Hom_{\Abr}(\cF_1|_p \tensor_R \Rbr,\cN_p \tensor_R \Rbr) = \Hom_{\Abr}(\oplus_{j=1}^{b_1} \Abr(a_{1,j}),\cN|_p\tensor_R \Rbr) = \\ 
\oplus_{j=1}^{b_1}{ \cN|_p \tensor_R \Rbr(-a_{1,j}) = \oplus_{j=1}^{b_1} W_\bt \tensor \Rbr(-a_{1,j}) }.
\end{multline*}
The right hand side is a finite dimensional graded vector space that only depends on $V_\bt$ and $W_\bt$.  This means that there is an $m$ only depending on $V_\bt$ and $W_\bt$ such that if $|j| > m$ then
\[ \Ext^1_A(\cM|_p,\cN|_p)_j = 0. \]
The Lemma then follows as explained above.
\end{proof}

\begin{proof}[Proof of Theorem \ref{thm:main.finiteness}]
As in the proof of Karroum's Theorem, fix a set of homogeneous generators $a_1,\dotsc,a_\gamma$ for $A$ over $\bk$ and let $\alpha = \max\{\deg(a_i) : 1 \leq i \leq \gamma \}$.  We will show that $\alpha_r = r \cdot \max\{ \alpha, \beta_1,\dotsc, \beta_r \} + 1$ satisfies the statement of Theorem \ref{thm:main.finiteness}.  

Consider a graded, MCM $A$-module $M$ and assume that $w(M) > \alpha_r$.  Then there exists $g_{min}(M) \leq i_0 < g_{max}(M)$ such that $(M/R_+ M)_{i_0+j} = 0$ for all $0 < j < \max\{\alpha,\beta_1,\dotsc,\beta_r\}$.  By Lemma \ref{lem:submodule} the $R$-submodule $M' = \sum_{i \leq i_0 }{ R \cdot M_i }$ is in fact an MCM $A$-submodule and $M/M'$ is also MCM.  We will now show that the choice of $\alpha_r$ guarantees that the extension
\[ 0 \to M' \to M \to M/M' \to 0 \]
splits.  Let $s = \max\{ r(M'), r(M/M')\}$.  Then by construction $g_{max}(M') < g_{min}(M/M') + \beta_s$.  Hence $\Ext^1_A(M/M',M')=0$ by Lemma \ref{lem:ext-degrees} and the above extension splits.
\end{proof}

\begin{proof}[Proof of Corollary \ref{cor:complete}]
Let $A$ be a non-negatively graded $\bk$-algebra and let $\Ah$ be its completion with respect to the irrelevant ideal.  By \cite{KMVDB}, every rigid MCM $\Ah$ module is the completion of a graded MCM $A$-module.  Since $\Ah$ has an isolated singularity, so does $A$.  By Theorem \ref{thm:main.finiteness}, there are finitely many isomorphism classes of indecomposable rigid, graded, MCM $A$ modules of each rank, up to shifting.  Hence there are finitely many isomorphism classes of indecomposable, rigid MCM $\Ah$-modules of each rank.  It is then immediate that there are only finitely many isomorphism classes of rigid, MCM $\Ah$-modules of each rank. 
\end{proof}


\section{Questions and conjectures}
Throughout this section $R$ is a Noetherian local ring. Our aim is to discuss the natural local analogue of our main Theorem \ref{mainThm} on rigid modules. We state it as a:  

\begin{conjecture}\label{f-rigid}
Let $R$ be a complete Noetherian local ring. Fix an integer $N$. Up to isomorphism, there are only finitely many rigid MCM $R$-modules of (Hilbert-Samuel) multiplicity at most $N$. 
\end{conjecture}

When $R$ is a domain, the conjecture can be formulated in terms of ranks. 

\begin{conjecture}\label{f-rank}
Let $R$ be a complete Noetherian local domain. Fix an integer $N$. Up to isomorphism, there are only finitely many rigid MCM $R$-modules of rank at most $N$. 
\end{conjecture}

There are several supporting pieces of evidence for these conjectures. It is true when $\dim(R) = 0$. Our Corollary \ref{cor:complete} establishes the conjectures for rings $R$ with an isolated singularity which are the completion of a non-negatively graded algebra. However, the general statements appear to be difficult. For instance, we do not know if they  hold even for one-dimensional Gorenstein or complete intersection domains. In fact, it may be true that over a one-dimensional complete intersection domain, a rigid $\MCM$ module is free, but it is known only when $R$ is a hypersurface (see \cite{Dao}, especially Section 9).

For the rest of this section we will point out a few links between the above conjectures and some recent works. Before moving on, we recall a relevant definition from \cite{IW10}. A reflexive  module $M$  is called \emphb{modifying} if $\End_R(M)$ is MCM. These modules have been studied intensely recently due to their connections to non-commutative (crepant) desingularizations.  The property of being modifying is closely related to rigidity. 

\begin{prop}\label{serre2}
Suppose that $R$ is normal and satisfies Serre's condition $(R_2)$, that is that $R_{\mathfrak{p}}$ is regular for primes $\mathfrak{p}$ of height at most two. Let $M$ be a reflexive $R$-module. If $M$ is modifying, then it is rigid. The converse is true when $\dim(R) \leq 3$. 
\end{prop}

\begin{proof}
This is well-known but we could not find a convenient reference. Let $M$ be  a modifying module. As $M$ is reflexive, if $\dim(R)\leq 2$ then $M$ is automatically free, hence rigid. Suppose $\dim(R)\geq 3$. By induction on dimension and localizing, $\Ext^1_R(M,M)$ is supported only at the maximal ideal, so is either $0$ or has depth $0$. Now take a projective cover of $M$: $0 \to \Omega M \to F \to M \to 0$ and apply $\Hom_R(-,M)$ yields: 
\[ 0 \to \Hom_R(M,M) \to  \Hom_R(F,M) \to \Hom_R(\Omega M, M) \to \Ext^1_R(M,M)\to 0. \]
Counting depths along this new exact sequence one sees that $\Ext^1_R(M,M) =  0$. The same sequence also implies the converse when $\dim(R) \leq 3$. 
\end{proof}

\begin{conjecture}\label{ser2}
Let $R$ be complete local Cohen-Macaulay ring satisfying Serre's condition $(R_2)$. The the set of MCM elements in the class group of $R$ is finite. 
\end{conjecture} 
As far as we know, this is open even when $\dim(R)=3$ and $R$ is a complete intersection singularity. When $R$ is a hypersurface and contains a field, it is known by \cite[Corollary 4.8]{DK} (see also Section 4 of that paper for some discussion of this conjecture). 

The point is, under the assumptions on $R$, the MCM elements in the class group are rigid. In fact, by basic property of the class group, such modules are modifying. 

\begin{prop} \label{f-modi}
Let $(R,\frak m)$ be a three dimensional complete local domain. Suppose for some element $t\in \frak m-\frak m^2$, $R/tR$ is a quotient singularity. If Conjecture \ref{f-rank} holds, then $R$ has only finitely many indecomposable rigid MCM modules, up to isomorphism. 
\end{prop}

\begin{proof}
By the proof of \cite[Prop 4.3]{DH}, we know that all indecomposable rigid modules have  rank bounded above by the maximal rank of the MCM modules over $R/tR$ (this number exists since $R/tR$ is a two dimensional quotient singularity, hence has only finitely many $\MCM$ modules up to isomorphism). Conjecture \ref{f-rank} then implies that there are only finitely many such modules up to isomorphism.  
\end{proof}

The above Proposition is related to some recent results in \cite{IW10, W}. Recall that if $R/tR$ is a Kleinian (Du Val, or simple) singularity for a general choice of $t$, then $R$ is called a \emphb{compound Du Val (cDV)} singularity. Also, recall from \cite{IW10} that a reflexive module is \emphb{maximal modifying} if $\add(M)= \{X \ | \End_R(M\oplus X) \in \MCM(R)\}$.  It has recently been shown that there are only finitely many modules that are indecomposable summands of maximal modifying modules which is a generator (has $R$ as a summand). The proof rests upon some sophisticated birational geometry and tilting theory. For example, it was shown (\cite[Theorem 4.9]{W}) that the basic maximal modifying generators corresponds bijectively to minimal models over $\spec(R)$. 

Since maximal modifying modules are automatically MCM and modifying, Propositions \ref{serre2} and \ref{f-modi} say that Conjecture \ref{f-rank} would imply such finiteness results when $R$ has an isolated singularity such that for some choice of $t$, $R/tR$  is a quotient singularity.


\bibliographystyle{alpha}
\bibliography{ds}
\end{document}